\documentclass[11pt,a4paper]{article}
\usepackage{epsf,epsfig,amsfonts,amsgen,amsmath,amstext,amsbsy,amsopn,amsthm,lineno}
\usepackage{color}
\usepackage{graphicx}
\newtheorem{theorem}{Theorem}

\theoremstyle{definition}

\newtheorem{claim}{Claim}

\newtheorem{conjecture}{Conjecture}

\begin{document}
\title
{\bf\Large Rainbow triangles in edge-colored graphs\thanks{Supported by NSFC (No.~11271300) and the Doctorate Foundation of Northwestern
Polytechnical University (No.~cx201202 and No.~cx201326).}}

\date{}
\author{\small Binlong Li, Bo Ning, Chuandong Xu and Shenggui Zhang\thanks{Corresponding author. E-mail address: sgzhang@nwpu.edu.cn} \\[2mm]
\small Department of Applied Mathematics, Northwestern Polytechnical University,\\
\small Xi'an, Shaanxi 710072, P.R.~China}
\maketitle

\begin{abstract}
Let $G$ be an edge-colored graph. The color degree of a vertex $v$ of $G$, is defined as the number of colors of the edges incident to $v$. The color number of $G$ is defined as the number of colors of the edges in $G$. A rainbow triangle is one in which every pair of edges have distinct colors. In this paper we give some sufficient conditions for the existence of rainbow triangles in edge-colored graphs in terms of color degree, color number and edge number. As a corollary, a conjecture proposed by Li and Wang (Color degree and heterochromatic cycles in edge-colored graphs, European J. Combin. 33 (2012) 1958--1964) is confirmed.
\medskip

\smallskip

\end{abstract}

\section{Introduction}

All graphs considered here are simple and finite. For terminology and notation not defined here, we refer the reader to Bondy and Murty \cite{Bondy_Murty}.

Let $G=(V,E)$ be a graph. We use $e(G)$ to denote the number of edges of $G$. An \emph{edge-coloring} of $G$ is a mapping $C: E\rightarrow \mathbb N$, where $\mathbb N$ is the set of natural numbers. We call $G$ an \emph{edge-colored graph} (or briefly, a {\em colored graph}) if it is assigned such an edge-coloring $C$ and use $C(G)$ to denote the set, and $c(G)$ the number (called the \emph{color number} of $G$), of colors of edges in $G$. For a vertex $v$ of $G$, the \emph{color degree} of $v$ in $G$ with the edge-coloring $C$, denoted by $d^{c}_{G}(v)$ (or briefly, $d^{c}(v)$), is defined as the number of colors of the edges incident to $v$.  A triangle in a colored graph is called \emph{rainbow} if every two of its edges have distinct colors.

In this paper, we mainly study the existence of rainbow triangles in colored graphs. Let $G$ be a colored graph on $n$ vertices. It follows from Tur\'{a}n's theorem that $G$ contains a triangle if $e(G)>\lfloor n^2/4\rfloor$. Thus $G$ contains a rainbow triangle if $c(G)>\lfloor n^2/4\rfloor$. This lower bound is sharp by considering the graph $G=K_{\lceil n/2\rceil,\lfloor n/2\rfloor}$ with edges assigned pairwise distinct colors.

Now we give two nontrivial conditions for the existence of rainbow triangles in colored graphs.

\begin{theorem}\label{th1}
Let $G$ be a colored graph on $n$ vertices. If $e(G)+c(G)\geq n(n+1)/2$, then $G$ contains a rainbow triangle.
\end{theorem}

\begin{theorem}\label{th2}
Let $G$ be a colored graph on $n$ vertices. If $\sum_{v\in V(G)}d^c(v)\geq n(n+1)/2$, then $G$ contains a rainbow triangle.
\end{theorem}

Let $G$ be a complete graph with vertex set $V(G)=\{v_1,v_2,\ldots,v_n\}$. For the edge $v_iv_j$, $1\leq i<j\leq n$, we assign the color $i$ to it. Then $e(G)+c(G)=\sum_{v\in V(G)}d^c(v)=n(n+1)/2-1$, and $G$ contains no rainbow triangles. This implies that the bounds of Theorems \ref{th1} and \ref{th2} are both sharp.

Li and Wang \cite{Li_Wang} conjectured that a colored graph $G$ on $n$ vertices contains a rainbow triangle if $d^{c}(v)\geq (n+1)/2$ for every vertex $v\in V(G)$. As a corollary of Theorem \ref{th2}, we can see that Li and Wang's conjecture is true.

\medskip
\noindent
{\bf Corollary 1}\footnote{During the revision of the paper, the authors learned that Li \cite{Li} had already proved this result in a recent paper (Rainbow $C_3$'s and $C_4$'s in edge-colored graphs, \emph{Discrete Math.}~313 (2013) 1893--1896).}
Let $G$ be a colored graph on $n$ vertices. If $d^{c}(v)\geq (n+1)/2$ for every vertex $v\in V(G)$, then $G$ contains a rainbow triangle.
\medskip

With more effort, we can prove the following stronger theorem.

\begin{theorem}\label{th3}
Let $G$ be a colored graph on $n$ vertices. If $d^{c}(v)\geq n/2$ for every vertex $v\in V(G)$ and $G$ contains no rainbow triangles, then $n$ is even and $G$ is the complete bipartite graph $K_{n/2,n/2}$, unless $G=K_4-e$ or $K_4$ when $n=4$.
\end{theorem}

Let $D=(V,A)$ be a digraph and $v$ be a vertex of $D$. We use $N_D^+(v)$ ($N_D^-(v)$) to denote the set of out-neighbors (in-neighbors), and $d_D^+(v)$ ($d_D^-(v)$), the out-degree (in-degree) of $v$ in $D$. For $S\subset V(D)$, we use $D[S]$ to denote the subdigraph induced by $S$. The \emph{out-component number} of $v$, denoted by $\omega_D^+(v)$, is the number of components of $D[N^+(v)]$. When no confusion occurs, we use $N^+(v)$, $N^-(v)$, $d^+(v)$, $d^-(v)$ and $\omega^+(v)$ instead of $N_D^+(v)$, $N_D^-(v)$, $d_D^+(v)$, $d_D^-(v)$ and $\omega_D^+(v)$, respectively. For two vertices $u,v\in V(D)$, we say that $u$ \emph{dominates} $v$ if $uv\in A(D)$. An \emph{orientation} of an (undirected) graph $G$ is a digraph obtained from $G$ by replacing each edge with one of the two possible arcs with the same ends. Such a digraph is called an \emph{oriented graph}.

The research of directed triangles in oriented graphs is closely related to that of rainbow triangles in colored graphs. Let $D$ be an oriented graph. We construct a colored graph as follows: Let $v$ be a vertex of $D$ and $H$ be a component of $D[N^+(v)]$. We assign one color to all the arcs from $v$ to the vertices in $H$. For two arcs with different tails, or with the same tail, say $v$, but with heads in different components of $D[N^+(v)]$, we assign distinct colors to them. We call the underlying graph of $D$ with this edge-coloring an \emph{associated colored graph} of $D$, and denote it by $G(D)$. One can see that $D$ contains a directed triangle if and only if $G(D)$ contains a rainbow triangle. We omit the details (the readers can find the proof in Section 2).

Let $G$ be an associated colored graph of an oriented graph $D$. Note that the color degree of $v$ in $G$ is equal to $d_D^-(v)+\omega_D^+(v)$. This implies that $\sum_{v\in V(G)}d^c(v)=\sum_{v\in V(D)}(d_D^-(v)+\omega_D^+(v))=e(G)+c(G)$. This is the reason why we consider the sum of edge number and color number for the existence of rainbow triangles in colored graphs.

Now we come back to digraphs. We use $a(D)$ to denote the number of arcs of a digraph $D$. In the following, we give two theorems concerning directed triangles corresponding to Theorems \ref{th2} and \ref{th3}, respectively.

\begin{theorem}\label{th4}
Let $D$ be an oriented graph on $n$ vertices. If $a(D)+\sum_{v\in V(D)}\omega^+(v)\geq n(n+1)/2$, then $D$ contains a directed triangle.
\end{theorem}

\begin{theorem}\label{th5}
Let $D$ be an oriented graph on $n$ vertices. If $d^-(v)+\omega^+(v)\geq n/2$ for every vertex $v\in V(D)$, then either $D$ contains a directed triangle or $n$ is even and $D$ is an orientation of $K_{n/2,n/2}$.
\end{theorem}

The proofs of Theorems \ref{th2} and \ref{th3} are heavily based on Theorems \ref{th4} and \ref{th5}, respectively.

The following conjecture concerning directed triangles, which is a special case of the famous Caccetta-H\"{a}ggkvist Conjecture, is still open.

\begin{conjecture}[Caccetta and H\"{a}ggkvist \cite{Caccetta_Haggkvist}]
Any oriented graph on $n$ vertices with minimum in-degree at least $n/3$ contains a directed triangle.
\end{conjecture}

Since this conjecture is difficult to prove, one may seek for the value $\alpha$ as small as possible such that every oriented graph on $n$ vertices with minimum in-degree at least $\alpha n$ contains a directed triangle. The best value of $\alpha$ known to us is $0.3435\cdots$ (See Lichiardopol \cite{Lichiardopol}). We list the following result due to Shen, which is used in our proof of Theorem \ref{th5}.

\begin{theorem}[Shen \cite{Shen}]\label{th6}
If $\alpha=3-\sqrt{7}=0.3542\cdots$, then any oriented graph on $n$ vertices with minimum in-degree at least $\alpha n$ contains a directed triangle.
\end{theorem}

\section{Proofs of the theorems}

\textbf{Proof of Theorem \ref{th1}.}

Suppose the contrary. Let $G$ be a counterexample with the smallest number of vertices, and then with the smallest number of edges.

\begin{claim}
$G$ contains two edges with the same color.
\end{claim}

\begin{proof}
It follows from Tur\'{a}n's theorem that there exists a triangle in $G$, which has two edges with the same color since $G$ has no rainbow triangle.
\end{proof}

\begin{claim}
$e(G)+c(G)=n(n+1)/2$.
\end{claim}

\begin{proof}
By Claim 1, let $e_1$ and $e_2$ be two edges with the same color. Then $e(G-e_1)=e(G)-1$ and $c(G-e_1)=c(G)$. If $e(G)+c(G)\geq n(n+1)/2+1$, then $e(G-e_1)+c(G-e_1)\geq n(n+1)/2$. Note that $G-e_1$ does not contain a rainbow triangle. Thus $G-e_1$ is a counterexample with fewer edges, a contradiction.
\end{proof}

Let $v$ be a vertex in $G$, and $s$ a color in $C(G)$. If all the edges with color $s$ are incident to $v$, then we call $s$ a color \emph{saturated} by $v$. We use $d^s(v)$ to denote the number of colors saturated by $v$.

\begin{claim}
$d(v)+d^s(v)\geq n+1$, for every $v\in V(G)$.
\end{claim}

\begin{proof}
Note that $e(G-v)=e(G)-d(v)$. If a color in $C(G)$ is not saturated by $v$, then it is also a color in $C(G-v)$. This implies that $c(G-v)=c(G)-d^s(v)$. If  $d(v)+d^s(v)\leq n$, then $$e(G-v)+c(G-v)=e(G)-d(v)+c(G)-d^s(v)\geq \frac{n(n-1)}{2}.$$ Note that $G-v$ does not contain a rainbow triangle. Thus $G-v$ is a counterexample with fewer vertices, a contradiction.
\end{proof}

\begin{claim}
$\sum_{v\in V(G)}d^s(v)\leq 2c(G)$, and the equality holds if and only if every two edges have distinct colors.
\end{claim}

\begin{proof}
Let $c$ be an arbitrary color in $C(G)$. Note that $c$ cannot be saturated by more than two vertices, and $c$ is saturated by exactly two vertices if and only if $c$ appears on only one edge. Thus we have $\sum_{v\in V(G)}d^s(v)\leq 2c(G)$, and the equality holds if and only if every two edges have distinct colors.
\end{proof}

By Claims 2, 3 and 4, we can get that
$$
n(n+1)\leq \sum_{v\in V(G)}(d(v)+d^s(v))\leq 2e(G)+2c(G)=n(n+1).
$$
This implies that $\sum_{v\in V(G)}(d(v)+d^s(v))=2e(G)+2c(G)$ and $\sum_{v\in V(G)}d^s(v)=2c(G)$. By Claim 4, every two edges have distinct colors,  contradicting to Claim 1.

The proof is complete. {\hfill$\Box$}

\noindent\textbf{Proof of Theorem \ref{th2}.}

The proof of this theorem is based on Theorem \ref{th4}, which will be proved later. Suppose that $G$ satisfies the condition of Theorem \ref{th2} but contains no rainbow triangles. Let $G'$ be a spanning subgraph of $G$ satisfying the condition of Theorem \ref{th2} with number of edges as small as possible.

\setcounter{claim}{0}
\begin{claim}
For each edge $uv\in E(G')$, one of the following is true:\\
(1) $C(uw)\neq C(uv)$ for $w\in N_{G'}(u)\backslash\{v\}$; or\\
(2) $C(wv)\neq C(uv)$ for $w\in N_{G'}(v)\backslash\{u\}$.
\end{claim}

\begin{proof}
If $C(uw)=C(uv)$ for some $w\in N_{G'}(u)\backslash\{v\}$, then the removal of the edge $uv$ does not reduce the color degree of $u$. If $C(wv)=C(uv)$ for some $w\in N_{G'}(v)\backslash\{u\}$, then the
removal of the edge $uv$ does not reduce the color degree of $v$. Since $G$ contains the fewest edges, either (1) or (2) holds.
\end{proof}

Now we give an orientation to $G'$ in such a way: for $uv\in E(G')$, if (1) of Claim 1 holds, then the orientation of the edge is from $v$ to $u$; if (2) holds, then the orientation is from $u$ to $v$; if both (1) and (2) hold, then we give the orientation arbitrarily. We denote the resulting oriented graph by $D$. By the construction of $D$, we have
\begin{claim}
If $uv\in A(D)$, then $C(uv)$ is different from the colors of every other arcs incident to $v$.
\end{claim}

\begin{claim}
Let $v$ be a vertex of $D$. If $x,y\in N^+(v)$ and $xy\in A(D)$, then $C(vx)=C(vy)$.
\end{claim}

\begin{proof}
By Claim 2, $C(vx)\neq C(xy)$ and $C(vy)\neq C(xy)$. If $C(vx)\neq C(vy)$, then $vxyv$ is a rainbow triangle in $G$, a contradiction.
\end{proof}

By applying Claim 3 repeatedly, we can conclude that if $H$ is a component of $D[N^+(v)]$, then the colors of the arcs from $v$ to all vertices in $H$ are the same.

\begin{claim}
$d^-(v)+\omega^+(v)\geq d_{G'}^c(v)$, for every vertex $v\in V(D)$.
\end{claim}

\begin{proof}
By Claim 2, every arc with head $v$ has the color different from the colors of the other arcs incident to $v$. By Claim 3, the arcs from $v$ to the vertices in the same component of $D[N^+(v)]$ have the same color. Hence $d^-(v)+\omega^+(v)\geq d_{G'}^{c}(v)$.
\end{proof}

By Claim 4, we have
$$
a(D)+\sum\limits_{v\in V(D)}\omega^+(v)=\sum\limits_{v\in V(D)}(d^-(v)+\omega^+(v))\geq\sum\limits_{v\in V(G')}d_{G'}^c(v)\geq \frac{n(n+1)}{2}.
$$ By Theorem 4, there is a directed triangle in $D$, say $uvwu$. By Claim 2, $C(uw)\neq C(uv)$, $C(uv)\neq C(vw)$ and $C(vw)\neq C(uw)$. Therefore, $uvwu$ is a rainbow triangle in $G$, a contradiction.

The proof is complete. {\hfill$\Box$}

\noindent\textbf{Proof of Theorem \ref{th3}.}

The proof of this theorem is based on Theorem \ref{th5}. Suppose that $G$ contains no rainbow triangles and $d^c(v)\geq n/2$ for every $v\in V(G)$. Let $G'$ be a spanning subgraph of $G$ satisfying the condition of Theorem \ref{th3} with number of edges as small as possible. As in the proof of Theorem \ref{th2}, we have

\setcounter{claim}{0}
\begin{claim}
For each edge $uv\in E(G')$, one of the following is true:\\
(1) $C(uw)\neq C(uv)$ for $w\in N_{G'}(u)\backslash\{v\}$; or\\
(2) $C(wv)\neq C(uv)$ for $w\in N_{G'}(v)\backslash\{u\}$.
\end{claim}

Now we give an orientation to $G'$ as in the proof of Theorem \ref{th2}, and similarly, we have

\begin{claim}
If $uv\in A(D)$, then $C(uv)$ is different from the colors of every other arcs incident to $v$.
\end{claim}

\begin{claim}
Let $v$ be a vertex of $D$. If $x,y\in N^+(v)$ and $xy\in A(D)$, then $C(vx)=C(vy)$.
\end{claim}

\begin{claim}
$d^-(v)+\omega^+(v)\geq d_{G'}^c(v)$, for every vertex $v\in V(D)$.
\end{claim}

By Claim 4, we have $d^-(v)+\omega^+(v)\geq n/2$ for every $v\in V(D)$. By Theorem \ref{th5}, either $D$ contains a directed triangle or $n$ is even and $D$ is an orientation of the complete bipartite graph $K_{n/2,n/2}$. If there is a directed triangle in $D$, then it is a rainbow triangle in $G$, a contradiction. Thus we assume that $n$ is even and $D$ is an orientation of $G'=K_{n/2,n/2}$.

For any vertex $v\in V(G')$, since $d_{G'}^c(v)\geq n/2$ and $d_{G'}(v)=n/2$, every pair of edges incident to $v$ have distinct colors. Note that $G$ is a spanning supergraph of $G'$. If $n=2$, then $G=K_2$. If $n=4$, then $G=K_{2,2},\ K_4-e$ or $K_4$. Now suppose that $n\geq 6$, and we will show that there are no edges in $E(G)\backslash E(G')$. If not, then we assume that $uv\in E(G)$ with $u,v$ in a same partition set of the bipartite graph $G'$. Let $x,y,z$ be three vertices in the other partition set of $G'$. Since $ux$, $uy$ and $uz$ have pairwise distinct colors, there are at least two edges in $\{ux,uy,uz\}$ with colors different from $uv$. Similarly, there are at least two edges in $\{vx,vy,vz\}$ with colors different from $uv$. Hence either $uvxu$, $uvyu$, or $uvzu$ is a rainbow triangle in $G$, a contradiction.

The proof is complete. {\hfill$\Box$}

\noindent\textbf{Proof of Theorem \ref{th4}.}

Let $G$ be the associated colored graph of $D$. We first prove the following claim.

\setcounter{claim}{0}
\begin{claim}
$D$ has a directed triangle if and only if $G$ has a rainbow triangle.
\end{claim}

\begin{proof}
If $D$ has a directed triangle, say $uvwu$, then by the definition of associated colored graphs, $C(uv)\neq C(vw)$, $C(vw)\neq C(wu)$ and $C(wu)\neq C(uv)$. Thus $uvwu$ is a rainbow triangle in $G$.

Conversely, suppose that $G$ contains a rainbow triangle, say $uvwu$. If $\{u,v,w\}$ does not induce a directed triangle, then there is a vertex, say $u$, dominating the other two vertices. But in this case, $v$ and $w$ are in the same component of $D[N^+(u)]$. By the definition of associated colored graphs, $C(uv)=C(uw)$, a contradiction.
\end{proof}

Note that $e(G)=a(D)$ and $c(G)=\sum_{v\in V(D)}\omega^+(v)$. We have
$$
e(G)+c(G)=a(D)+\sum_{v\in V(D)}\omega^+(v)\geq \frac{n(n+1)}{2}.
$$
By Theorem \ref{th1}, there is a rainbow triangle in $G$; and by Claim 1, there is a directed triangle in $D$.

The proof is complete. {\hfill$\Box$}

\noindent\textbf{Proof of Theorem \ref{th5}.}

We prove the theorem by induction on $n$. Since the result is trivially true when $n=2,3$, we assume that $n\geq 4$. If $d^-(v)\geq \alpha n$ for every vertex $v\in V(D)$, where $\alpha=3-\sqrt{7}=0.3542\cdots$, then there is a directed triangle by Theorem \ref{th6}. Thus we suppose that there is a vertex $v$ such that
\begin{align}
d^-(v)<\alpha n.
\end{align}

Noting that $d^-(v)+\omega^+(v)\geq n/2$, we have
\begin{align}
\omega^+(v)\geq \frac{n}{2}-d^-(v).
\end{align}

\setcounter{claim}{0}
\begin{claim}
There is a component of $D[N^+(v)]$ with only one vertex.
\end{claim}

\begin{proof}
We use $b(v)$ to denote the number of vertices which are not adjacent to $v$.

Suppose that every component of $D[N^+(v)]$ has at least two vertices. Then
$$n=d^-(v)+d^+(v)+1+b(v)\geq d^-(v)+2\omega^+(v)+1+b(v),
$$ and by (2),
$$b(v)\leq n-d^-(v)-2\omega^+(v)-1\leq n-d^-(v)-2(\frac{n}{2}-d^-(v))-1.$$
That is,
\begin{align}
b(v)\leq d^-(v)-1.
\end{align}

Let $H$ be the subdigraph of $D$ induced by $N^-(v)$. If for every vertex $u\in V(H)$, $d_{H}^-(u)\geq \alpha d^-(v)=\alpha|V(H)|$, then by Theorem \ref{th6}, there is a directed triangle in $H$. Thus we assume that there is a vertex $u\in V(H)$ such that $d_{H}^-(u)<\alpha d^-(v)$.

First for every $w\in N^+(v)$, $wu\notin A(D)$; otherwise $uvwu$ is a directed triangle. Since $uv\in A(D)$, all the out-neighbors of $u$ in $\{v\}\cup N^-(v)\cup N^+(v)$ are in a same component of $D[N^+(u)]$. Every vertex not adjacent to $v$ contributes at most one to $d^-(u)+\omega^+(u)$. Thus we have
$$
d^-(u)+\omega^+(u)\leq d_{H}^-(u)+1+b(v)<\alpha d^-(v)+1+b(v).
$$ Since $d^-(u)+\omega^+(u)\geq n/2$, we have
\begin{align}
b(v)>\frac{n}{2}-1-\alpha d^-(v).
\end{align}

Combining (3) with (4), we have $n/2-1-\alpha d^-(v)<d^-(v)-1$, and
\begin{align*}
d^-(v)>\frac{n}{2(1+\alpha)}>\alpha n
\end{align*}
(noting that $2\alpha(1+\alpha)=0.9594\cdots <1$), contradicting to (1).
\end{proof}

Now let $w$ be an isolated vertex of $D[N^+(v)]$, and let $D'=D-\{v,w\}$.

\begin{claim}
For every vertex $u\in V(D')$, $d_{D'}^-(u)+\omega_{D'}^+(u)\geq d^-(u)+\omega^+(u)-1$.
\end{claim}

\begin{proof}
First we assume that $u\in N^-(v)$. Note that $wu\notin A(D)$; otherwise $uvwu$ will be a directed triangle. We have $d_{D'}^-(u)=d^-(u)$. If $uw\in A(D)$, then $v$ and $w$ are in the same component of $D[N^+(u)]$. Since the removal of $\{v,w\}$ does not change the components of $D[N^+(u)]$ not containing $v$, we have $\omega_{D'}^+(u)\geq \omega^+(u)-1$, and then $d_{D'}^-(u)+\omega_{D'}^+(u)\geq d^-(u)+\omega^+(u)-1$.

Next we assume that $u\in N^+(v)\backslash\{w\}$. Since $w$ is an isolated vertex of $D[N^+(v)]$, it is not adjacent to $u$. This implies that $d_{D'}^-(u)=d^-(u)-1$ and $\omega_{D'}^+(u)=\omega^+(u)$. Thus $d_{D'}^-(u)+\omega_{D'}^+(u)=d^-(u)+\omega^+(u)-1$.

At last, we assume that $u$ is not adjacent to $v$. If $u$ and $w$ are not adjacent to each other, then the removal of $\{v,w\}$ does not change the in- and out-neighbors of $u$. If $wu\in A(D)$, then $d_{D'}^-(u)=d^-(u)-1$ and $\omega_{D'}^+(u)=\omega^+(u)$. If $uw\in A(D)$, then $d_{D'}^-(u)=d^-(u)$, and the removal of $\{v,w\}$ does not change the components of $D[N^+(u)]$ not containing $w$. In any case, we have $d_{D'}^-(u)+\omega_{D'}^+(u)\geq d^-(u)+\omega^+(u)-1$.
\end{proof}

By induction hypothesis, $D'$ contains a directed triangle or $n$ is even and $D'$ is an orientation of $K_{n/2-1,n/2-1}$. If $D'$ contains a directed triangle, then it is also a directed triangle in $D$. Now we assume that $n$ is even and $D'$ is an orientation of $K_{n/2-1,n/2-1}$. Let $V(D')=X\cup Y$, where $X$ and $Y$ are two partition sets of the bipartite graph $D'$.

\begin{claim}
For every vertex $u\in V(D)\backslash\{v,w\}$, $u$ is adjacent to exactly one vertex of $\{v,w\}$.
\end{claim}

\begin{proof}
If $u$ is adjacent to neither $v$ nor $w$, then $d^-(u)+\omega^+(u)\leq d^-(u)+d^+(u)=n/2-1$, a contradiction. This implies that any vertex in $V(D)\backslash\{v,w\}$ is adjacent to at least one vertex in $\{v,w\}$.

Now suppose the contrary that $u$ is adjacent to both $v$ and $w$. If $vu\in A(D)$, then $w$ and $u$ are in the same component of $D[N^+(v)]$, contradicting to that $w$ is an isolated vertex of $D[N^+(v)]$. Thus we assume that $uv\in A(D)$. If $wu\in A(D)$, then $uvwu$ is a directed triangle. Thus we assume that $uw\in A(D)$. Without loss of generality, we assume that $u\in X$.

Let $y\in Y$. We claim that $yu\in A(D)$. Suppose the contrary that $uy\in A(D)$. Since $y$ is adjacent to either $v$ or $w$, $\{y,v,w\}$ is contained in a same component of $D[N^+(u)]$. Note that $u$ is adjacent to $n/2-2$ vertices other that $y,v$ and $w$. This implies that $d^-(u)+\omega^+(u)\leq n/2-1$, a contradiction. Thus as we claimed, $yu\in A(D)$.

If $vy\in A(D)$ or $wy\in A(D)$, then $uvyu$ or $uwyu$ is a directed triangle. Thus we assume that $vy\notin A(D)$ and $wy\notin A(D)$. Note that $v,w$ (if dominated by $y$) and $u$ are in a same component of $D[N^+(y)]$, and $y$ is adjacent to $n/2-2$ vertices other that $u,v$ and $w$. This implies that $d^-(y)+\omega^+(y)\leq n/2-1$, a contradiction.
\end{proof}

Since $d^-(v)+d^+(v)\geq d^-(v)+\omega^+(v)\geq n/2$ and $d^-(w)+d^+(w)\geq d^-(w)+\omega^+(w)\geq n/2$, by Claim 3, we can see that $d^-(v)+d^+(v)=n/2$, $d^-(w)+d^+(w)=n/2$. This implies that every vertex in $D$ is adjacent to exactly $n/2$ vertices. We claim that for every $u\in V(D)$, $N^+(u)$ is an independent set. If not, then there is a component of $D[N^+(u)]$ containing at least two vertices. This implies that $\omega^+(u)<d^+(u)$ and $d^-(u)+\omega^+(u)<n/2$, a contradiction.

Now we claim that $v$ cannot be adjacent to one vertex $x\in X$ and one vertex $y\in Y$. Suppose not. If $\{x,y,v\}$ does not induce a directed triangle, then there is a vertex, say $x$, dominating the other two vertices. But in this case, $N^+(x)$ is not an independent set, a contradiction.

Without loss of generality, we assume that $v$ is not adjacent to any vertex in $Y$ and then adjacent to all the vertices in $X$. By Claim 3, $w$ is not adjacent to any vertex in $X$ and adjacent to all the vertices in $Y$. Thus $D$ is an orientation of $K_{n/2,n/2}$.

The proof is complete. {\hfill$\Box$}

{\bf Acknowledgements.} The authors are particularly grateful to two anonymous referees for their extensive comments that considerably improved the paper.


\end{document}